\documentclass[11pt,a4paper,reqno]{amsart}
\usepackage{aliascnt}
\usepackage{hyperref}
\usepackage{amsfonts}
\usepackage{mathrsfs}
\usepackage{amsmath}
\usepackage{amsmath,graphics}
\usepackage{amssymb}
\usepackage{indentfirst,latexsym,bm,amsmath,amssymb,amsthm,graphicx,fancyhdr,float,color}
\IfFileExists{pstricks.sty}{\usepackage{pstricks}}{}
\usepackage[all]{xy}
\usepackage{amsthm}
\usepackage{amscd}
\usepackage[all]{hypcap}

\usepackage{enumitem,array,needspace}
\usepackage{xcolor}
\newcommand{\correction}[1]{#1}

\newcommand{\mathcorrection}[1]{#1}
\usepackage{hyperref}
\numberwithin{equation}{section}

\newtheorem{theorem}{Theorem}[section]

\newaliascnt{lemma}{theorem}
\newtheorem{lemma}[lemma]{Lemma}
\aliascntresetthe{lemma}

\newaliascnt{conjecture}{theorem}

\aliascntresetthe{conjecture}

\newaliascnt{proposition}{theorem}
\newtheorem{proposition}[proposition]{Proposition}
\aliascntresetthe{proposition}

\newaliascnt{corollary}{theorem}
\newtheorem{corollary}[corollary]{Corollary}
\aliascntresetthe{corollary}

\newaliascnt{problem}{theorem}

\aliascntresetthe{problem}

\newaliascnt{question}{theorem}

\aliascntresetthe{question}

\newaliascnt{claim}{theorem}

\aliascntresetthe{claim}

\theoremstyle{definition}

\newaliascnt{definition}{theorem}

\aliascntresetthe{definition}

\newaliascnt{example}{theorem}
\newtheorem{example}[example]{Example}
\aliascntresetthe{example}

\theoremstyle{remark}

\newaliascnt{remark}{theorem}
\newtheorem{remark}[remark]{Remark}
\aliascntresetthe{remark}

\newaliascnt{remarks}{theorem}

\aliascntresetthe{remarks}

\newcommand{\F}{\mathcal F}
\newcommand{\G}{\mathcal G}
\newcommand{\OO}{\mathcal O}
\newcommand{\PP}{\mathbb P}
\newcommand{\CC}{\mathbb C}
\newcommand{\ZZ}{\mathbb Z}
\newcommand{\QQ}{\mathbb Q}
\newcommand{\vol}{\operatorname{vol}}
\newcommand{\Alb}{\operatorname{Alb}}
\newcommand{\Supp}{\operatorname{Supp}}
\newcommand{\Sing}{\operatorname{Sing}}
\newcommand{\tang}{\operatorname{tang}}
\newcommand{\CS}{\operatorname{CS}}
\newcommand{\lcm}{\operatorname{lcm}}
\newcommand{\Kod}{\operatorname{Kod}}
\setlist[enumerate]{label=\textup{(\roman*)},leftmargin=2.2em,itemsep=2pt,topsep=4pt}
\def\Alb{\mathrm{Alb}}

\begin{document}

	\title{The volume of foliated surfaces with geometric genus one}	
	
	\author{Zhiming Guo}
	
	\address{Shanghai Center for Mathematical Sciences,
  Fudan University, Shanghai 200438, China}
	
	\email{z\_guo@fudan.edu.cn}
	
		\author{Bokai Hu}
	
	\address{School of Mathematical Sciences,  Key Laboratory of MEA (Ministry of Education) \& Shanghai Key Laboratory of PMMP,  East China Normal University, Shanghai 200241, China}
	
	\email{52280155002@stu.ecnu.edu.cn}

	\author{Xin L\"u}
	
	\address{School of Mathematical Sciences,  Key Laboratory of MEA (Ministry of Education) \& Shanghai Key Laboratory of PMMP,  East China Normal University, Shanghai 200241, China}
	
	\email{xlv@math.ecnu.edu.cn}
	
	\thanks{This work is supported by the National Natural Science Foundation of China, Scientific Research
		Innovation Capability Support Project for Young Faculty (No. SRICSPYF-ZY2025161),
		Shanghai Pilot Program for Basic Research (No. TQ20240202), Science and Technology Commission
		of Shanghai Municipality (No. 22DZ2229014).}
	
	
	
	
	
\begin{abstract}
We prove the sharp volume bound $\vol(\F)\ge1/6$ for complex foliated surfaces of general type with positive geometric genus. We also study the equality case and apply the bound to the algebraic integrability and finiteness of small volumes on irregular surfaces.
\end{abstract}
\maketitle

\section{Introduction}\label{sec:intro}
We work over the complex number field $\mathbb{C}$.
The main purpose of the paper is to establish a sharp lower bound on the volume of a foliation of general type with geometric genus one.

Let $(S,\F)$ be a foliated surface; namely $S$ is a smooth projective surface and $\F$ is determined by a saturated invertible subsheaf $T_{\F}\subset T_S$ closed under the Lie bracket.
The geometric genus and volume of $\F$ are
the birational invariants
\[
 p_g(\F)=\dim H^0(X,K_{\G}),\qquad
 \vol(\F)=\limsup_{m\to\infty}
 \frac{\dim H^0(X,mK_{\G})}{m^2/2},
\]
computed on any reduced model $(X,\G)$ of $(S,\F)$
\cite[Section~2.3]{LT24}. They do not depend on the choice of the
reduced model. The foliation is of general type if $\vol(\F)>0$.
We refer to \autoref{sec:chains} for the definitions and basic facts about foliated surfaces.

A basic question in the classification of foliated surfaces of general type is whether the volumes of foliations of general type
admit a uniform positive lower bound;
see \cite[Question~4]{HL21} and \cite[\S\,3]{cas21}.
For algebraically integrable foliations, the index
bound of Pereira--Svaldi \cite[Proposition~5.3]{PS19} gives a positive
lower bound depending on the genus of a general leaf. The existence of
a lower bound independent of the leaf genus, and of a uniform lower
bound for arbitrary foliations of general type, remains open.

L\"u--Tan \cite[Theorem~1.4]{LT24} established the Noether-type
inequalities for foliated surfaces of general type, based on the canonical maps.
As an application, they proved the following sharp lower bound \cite[Corollary~1.5 and Remark~1.6]{LT24}:
\begin{equation}\label{eq:LT-intro}
 \vol(\F)\ge\frac12,\qquad\qquad\text{if~}p_g(\F)\ge2.
\end{equation}

The case $p_g(\F)\le1$ requires a different argument, since there is no nonconstant canonical map. As in the study of volumes of varieties \cite{CC08,CC15,CHP21}, different methods are needed.
Our main result settles the problem in the case where $p_g(\F)=1$, 
and as a consequence gives a uniform positive lower bound for foliated surfaces of general type with positive geometric genus:
 
\begin{theorem}\label{thm:lower-bound}
Let $(S,\F)$ be a foliated surface of general type with $p_g(\F)>0$. Then
\begin{equation*}
 \vol(\F)\ge\frac16.
\end{equation*}
\end{theorem}
The lower bound in the above theorem is sharp. In fact,
we construct in \autoref{sec:examples} two foliated surfaces of general type with $p_g(\F)=1$ and volume $1/6$; see \autoref{thm:plane-family} with $n=2$ and \autoref{ex:invariant-canonical}.


The main idea of the proof of \autoref{thm:lower-bound}, motivated by Pereira--Svaldi \cite{PS19},
is to study $P\cdot A$ for the components $A$ of an effective canonical divisor, where $P$ is the nef part of the Zariski decomposition of $K_{\F}$. Invariant components satisfy a uniform intersection bound; non-invariant components are controlled by the tangency formula and the integrality of the coefficients along the negative chains.
For a foliation induced by a fibration of genus at least two,
Pereira--Svaldi \cite[Lemma~5.2]{PS19} proved the bound
$P\cdot A\ge1/42$ when $P\cdot A>0$ and $A$ is a tail, namely an
invariant curve meeting $\Supp(N)$ but not contained in it.
Extending their result, we obtain the following.


\begin{proposition}\label{thm:invariant-gap}
Let $(S,\F)$ be a reduced foliated surface with $K_{\F}$ pseudo-effective, and let $K_{\F}=P+N$ be its Zariski
decomposition. For every $\F$-invariant irreducible curve $A\subset S$,
\[
 \text{either~}~ P\cdot A=0\qquad\text{or~}~ P\cdot A\ge\frac16.
\]
Moreover, the constant $1/6$ is optimal.
\end{proposition}

The equality $P\cdot A=\frac16$ is realized in \autoref{ex:invariant-canonical}.
We remark that in the above proposition, there is no assumption on the geometric genus of the foliation.
It gives the sharp bound $1/6$ and applies
to all invariant curves of a reduced foliation with pseudo-effective
canonical divisor.
On the relatively minimal model of a foliated surface of general type, McQuillan's theorem classifies the curves with $P\cdot A=0$ and shows
that they are invariant \cite[III.3.2 and Remark~III.2.2]{McQ08}; see
also \cite[Proposition~2.15 and Theorem~2.16]{CF18}.
The equality $P\cdot A=1/6$ is characterized in \autoref{prop:invariant-intersection}.
The invariance hypothesis is essential: \autoref{rem:noninvariant-no-gap} gives non-invariant curves with positive intersections tending to zero.


Our method applies also to the case when one of the pluri-geometric genus $p_n(\F):=\dim H^0(S,nK_{\F})>0$ for some $n\geq 2$.
We will deal with this in a forthcoming paper.
It is natural to ask what happens when the equality holds in
\autoref{thm:lower-bound}.

\begin{theorem}\label{thm:equality-intro}
Let $(S,\F)$ be relatively minimal with $p_g(\F)>0$ and
$\vol(\F)=1/6$. Then $p_g(\F)=1$. Let $D$ be its unique effective
canonical divisor and let $D=P+N$ be its Zariski decomposition with $P$ nef.
Then exactly one of the following holds:
\begin{enumerate}
\item every component of $D$ is $\F$-invariant whose normalization is rational, and the surface $S$ is regular, i.e., $q(S)=0$;
\item the part of $D$ outside $\mathrm{Supp}(N)$ is a single smooth curve $C$, everywhere
transverse to $\F$. Moreover, the coefficient of $C$ in $D$ is exactly one.
\end{enumerate}
Both cases occur on rational surfaces.
\end{theorem}
The precise configurations of $D$ are given in \autoref{thm:equality-core}. The first case is realized in
\autoref{ex:invariant-canonical}, and the second case in \autoref{thm:plane-family} with $n=2$.

As an application, we study the Poincar\'e problem concerning the algebraic
integrability. A foliation is called algebraically integrable if it admits
a rational first integral.
The Poincar\'e problem asks for criteria to determine whether a foliation is algebraically integrable.
This problem has been studied extensively,
see, for instance, \cite{Dar78,Jou78,Ghy00,Per02,PS19,LT24}.
Thanks to \autoref{thm:lower-bound},
we are able to give a positive answer to the Poincar\'e problem for
foliations on irregular surfaces with small volume.
 
\begin{corollary}\label{cor:alb-intro}
Let $(S,\F)$ be a foliated surface of general type with $q(S)>0$. If
\[
 \vol(\F)<\frac16,
\]
then $p_g(\F)=0$ and $\F$ is induced by the Albanese fibration of $S$.
In particular, $\F$ is algebraically integrable.
\end{corollary}


Finally, we consider the possible volumes in this range.
Han--Jiao--Li--Liu \cite[Theorem~1.2]{HJLL24} proved discreteness of
the volumes of algebraically integrable foliations with canonical
singularities when the rank and the volume of a general leaf are
fixed.
For an algebraically integrable foliation of general type on a surface, the leaf volume is $2g-2$, where $g\ge2$ is the genus of a general leaf. In our setting, the general leaf is an Albanese fiber. Its genus is bounded when the surface volume $\vol(S)$ is bounded, by boundedness of surfaces of general type and deformation invariance of the Albanese fiber genus; see \cite{Cat15} and \autoref{lem:albanese-genus}. Together with the common denominator established in \autoref{prop:denominator}, this gives the following result.

\begin{theorem}\label{thm:finite-intro}
Fix a positive integer $c$. There exists a number $\varepsilon_c>0$ such
that every foliation $\F$ of general type on a smooth projective irregular
surface $S$ whose volume $\vol(S)\le c$ satisfies
\[
 \vol(\F)\ge\varepsilon_c.
\]
Moreover, the set of possible values $\vol(\F)\in(0,1/6)$ for these
foliations is finite.
\end{theorem}

\vspace{2mm}
The paper is organized as follows. In \autoref{sec:chains}, we
recall some basic facts about the foliated surfaces.
In \autoref{sec:invariant}, we prove \autoref{thm:invariant-gap} and characterize equality for invariant intersections.
The lower bound in \autoref{thm:lower-bound} is proved in \autoref{sec:lower-bounds}.
In \autoref{sec:equality}, we study the equality case.
\autoref{sec:albanese-geography} is devoted to studying the foliations on
irregular surfaces, the proofs of \autoref{cor:alb-intro} and \autoref{thm:finite-intro}, and the remaining uniform lower bound problems. Finally, in
\autoref{sec:examples}, we construct the examples mentioned above.

\section{Preliminaries}\label{sec:chains}

We recall some basic facts about foliated surfaces, and refer to \cite{Bru04,McQ08} and \cite[Section~2]{LT24} for more details.

Let $S$ be a smooth projective surface over $\CC$. A foliation $\F$ on
$S$ is a saturated invertible subsheaf $T_{\F}\subset T_S$, so that
$T_S/T_{\F}$ is torsion-free. The line bundle $T_{\F}$ is the tangent
bundle of $\F$, and its dual $K_{\F}=T_{\F}^{*}$ is the canonical
bundle. Locally, $\F$ is generated by a holomorphic vector field
\[
 v=a(x,y)\frac{\partial}{\partial x}
   +b(x,y)\frac{\partial}{\partial y}
\]
with isolated zeros. Its zero set is the singular locus $\Sing(\F)$;
local generators differ by multiplication by \correction{a nowhere-vanishing
holomorphic function}.

An irreducible curve $A\subset S$ is $\F$-invariant if $\F$ is tangent
to $A$ at a general smooth point. The foliation is algebraically
integrable if it admits a rational first integral, and transcendental
otherwise. After resolving a first integral and taking its Stein
factorization, an algebraically integrable foliation is induced by a
fibration $f:X\to B$. Here a fibration means a surjective morphism to
a smooth projective curve with connected fibers; its genus is that of
a general fiber.

Let $p\in\Sing(\F)$, and let $\lambda_1,\lambda_2$ be the eigenvalues
of the linear part of a local generator at $p$. The singularity is
reduced if, after reordering,
$\lambda_2\ne0$ and $\lambda_1/\lambda_2\notin\QQ_{>0}$.
It is non-degenerate if both eigenvalues are non-zero, and a saddle-node
if exactly one is non-zero. A foliation is reduced if all its
singularities are reduced. By Seidenberg's theorem \cite{Sei68},
every foliated surface admits a reduced model \correction{obtained by a finite
sequence of blow-ups}.

The Kodaira dimension of $\F$ is $\Kod(\F)=\Kod(K_{\G})$, computed
on a reduced model $(X,\G)$ of $(S,\F)$. It is independent of the
choice of reduced model \cite[Section~2.3]{LT24}. The foliation is of
general type if $\Kod(\F)=2$, equivalently if $\vol(\F)>0$.

For a reduced foliation, a smooth rational curve $E$ with $E^2=-1$
is $\F$-exceptional if its contraction produces a reduced foliation.
A foliated surface is relatively minimal if its foliation is reduced
and it contains no $\F$-exceptional curve. Every foliated surface
admits a relatively minimal model \cite[Chapter~5]{Bru04}.

Suppose that $K_{\F}$ is pseudo-effective on a relatively minimal model, and write
\begin{equation*}
 K_{\F}=P+N.
\end{equation*}
The support of $N$ is a disjoint union of maximal $\F$-chains, and $\lfloor N\rfloor=0$ \cite[Chapter~8]{Bru04}; see also
\cite{McQ08} and \cite[Theorem~2.9]{CF18}. Such a chain has the form
\[
 T=C_1+\cdots+C_r,\qquad C_j^2=-e_j,\quad e_j\ge2,
\]
where the $C_j$ are smooth $\F$-invariant rational curves, consecutive curves meet transversely at one point, and there are no other intersections. In particular, $T$ is a Hirzebruch--Jung string. We write $[-e_1,\ldots,-e_r]$ for this ordered chain; the entries record self-intersection numbers. Its canonical intersections and indices are
\begin{equation}\label{eq:chain-int}
 K_{\F}\cdot C_1=-1,\quad K_{\F}\cdot C_j=0\ (j\ge2),\qquad
 Z(\F,C_1)=1,\quad Z(\F,C_j)=2\ (j\ge2).
\end{equation}
All singularities on the chain are reduced and non-degenerate.

Fix a maximal chain $T=C_1+\cdots+C_r$.  Put
\begin{equation*}
  M_T=(-C_i\cdot C_j)_{i,j}
  =
  \begin{pmatrix}
    e_1&-1&0&\cdots&0\\
    -1&e_2&-1&\ddots&\vdots\\
    0&-1&e_3&\ddots&0\\
    \vdots&\ddots&\ddots&\ddots&-1\\
    0&\cdots&0&-1&e_r
  \end{pmatrix},
  \qquad \det M_T>0.
\end{equation*}
Define
\begin{align}
  \mu_0=0,\quad\mu_1=1,
  \qquad \mu_{j+1}&=e_j\mu_j-\mu_{j-1},
  \label{eq:mu}\\
  \xi_{r+1}=0,\quad\xi_r=1,
  \qquad \xi_{j-1}&=e_j\xi_j-\xi_{j+1}.\notag
\end{align}
Thus $\mu_j$ is the determinant of the part of the chain strictly before
$C_j$, and $\xi_j$ is the determinant of the part strictly after $C_j$.
In particular,
\[
  \mu_{r+1}=\xi_0=\det M_T.
\]
The recurrence and $e_j\ge2$ imply
\begin{equation}\label{eq:mu-increasing}
  0=\mu_0<\mu_1<\cdots<\mu_{r+1},
  \qquad
  \gcd(\mu_r,\mu_{r+1})=1.
\end{equation}
Indeed, $\mu_{j+1}\ge2\mu_j-\mu_{j-1}>\mu_j$ whenever
$\mu_j>\mu_{j-1}$, and
$\gcd(\mu_{j+1},\mu_j)=\gcd(\mu_j,\mu_{j-1})$.
The analogous strict inequalities hold for the $\xi_j$ in reverse order.
In particular, $\mu_{r+1}\ge r+1\ge2$.

\begin{lemma}[Inverse matrix]\label{lem:inverse}
For $i\le j$,
\begin{equation*}
  (M_T^{-1})_{ij}=\frac{\mu_i\xi_j}{\mu_{r+1}},
\end{equation*}
and $M_T^{-1}$ is symmetric.  Equivalently,
\begin{equation}\label{eq:inverse-matrix}
  M_T^{-1}=\frac1{\mu_{r+1}}
  \begin{pmatrix}
    \mu_1\xi_1&\mu_1\xi_2&\cdots&\mu_1\xi_r\\
    \mu_1\xi_2&\mu_2\xi_2&\cdots&\mu_2\xi_r\\
    \vdots&\vdots&\ddots&\vdots\\
    \mu_1\xi_r&\mu_2\xi_r&\cdots&\mu_r\xi_r
  \end{pmatrix}.
\end{equation}
All its entries are positive.
\end{lemma}

\begin{proof}
The recurrences give
\begin{equation}\label{eq:det-identity}
  \mu_{j+1}\xi_j-\mu_j\xi_{j+1}=\mu_{r+1}.
\end{equation}
Indeed, the left-hand side is independent of $j$, and its value at
$j=0$ is
$\mu_1\xi_0-\mu_0\xi_1=\mu_{r+1}$.

Fix $j$ and define
\[
  v_i=
  \begin{cases}
    \mu_i\xi_j/\mu_{r+1},&i\le j,\\
    \mu_j\xi_i/\mu_{r+1},&i\ge j.
  \end{cases}
\]
The recurrences show that $(M_T\mathbf v)_i=0$ for $i\ne j$, while
\eqref{eq:det-identity} gives $(M_T\mathbf v)_j=1$.  Thus $\mathbf v$ is the
$j$-th column of $M_T^{-1}$.
\end{proof}

\begin{corollary}\label{cor:negative-coefficients}
If $N_T=\sum_j b_jC_j$ is the part of $N$ supported on $T$, then
\begin{equation}\label{eq:b-formulas}
 b_j=\frac{\xi_j}{\mu_{r+1}},\qquad
 b_r=\frac1{\mu_{r+1}},\qquad
 (M_T^{-1})_{jj}=\mu_jb_j,\qquad
 (M_T^{-1})_{rj}=\frac{\mu_j}{\mu_{r+1}}.
\end{equation}
\end{corollary}
\begin{proof}
The equalities $P\cdot C_j=0$ and \eqref{eq:chain-int} give
$M_T\mathbf b=\mathbf e_1$, where $\mathbf b=(b_j)^t$ and
$\mathbf e_1$ is the first standard basis vector.
Thus $\mathbf b$ is the first column of $M_T^{-1}$, and the assertions
follow from \autoref{lem:inverse}.
\end{proof}

For an $\F$-invariant curve $A$, we use $Z(\F,A,p)$ and $\CS(\F,A,p)$ for the local $Z$-index and Camacho--Sad index, and write $Z(\F,A)$ and $\CS(\F,A)$ for their sums. We write $p_a(A)$ for the arithmetic genus. \correction{For a curve $A$ that is not $\F$-invariant}, $\tang(\F,A)$ is the total tangency number. The index and tangency formulas are
\begin{align}\label{eq:indices}
 K_{\F}\cdot A&=2p_a(A)-2+Z(\F,A),\qquad \CS(\F,A)=A^2,\\
 \label{eq:tangency}
 \tang(\F,A)&=K_{\F}\cdot A+A^2\ge0.
\end{align}
\correction{For a reduced foliation, the local $Z$-indices} of an invariant curve are non-negative. On a smooth invariant branch through a singularity they are positive. We use the conventions of \cite{Bru04,LT24}.

\begin{lemma}\label{lem:terminal}
Let $T=C_1+\cdots+C_r$ be a maximal chain, and let $A$ be an invariant curve
not contained in $T$. If $A$ meets $T$, it meets only $C_r$, transversely
at a single point $p_T$. The curve $A$ is smooth at $p_T$, and
\begin{equation*}
  Z(\F,A,p_T)=1,
  \qquad
  \CS(\F,A,p_T)=-\frac{\mu_r}{\mu_{r+1}},
  \qquad
  N_T\cdot A=\frac1{\mu_{r+1}},
\end{equation*}
where $N_T=\sum_jb_jC_j$. In particular, the Camacho--Sad index is a
reduced fraction, with $0<\mu_r<\mu_{r+1}$.
\end{lemma}

\begin{proof}
Distinct invariant curves can intersect only at a foliation singularity.
At a reduced non-degenerate singularity there are exactly two smooth
transverse separatrices, each with local $Z$-index one. At
$C_j\cap C_{j+1}$, these are $C_j$ and $C_{j+1}$.
By \eqref{eq:chain-int}, for $r\ge2$, there are no other singularities on
$C_1,\ldots,C_{r-1}$, and $C_r$ has exactly one further singularity $p_T$.
For $r=1$, let $p_T$ be the unique singularity on $C_1$.
Thus the curve $A$ meets $T$ only at $p_T$, where it coincides locally
with the other separatrix and has local $Z$-index one.

Put $p_j=C_j\cap C_{j+1}$ for $1\le j<r$, and $p_r=p_T$.
The Camacho--Sad formula on $C_1$ gives the index $-e_1=-\mu_2/\mu_1$
at $p_1$. Since the indices on the two separatrices are reciprocal,
induction along the chain gives
\[
 \CS(\F,C_j,p_j)
 =-e_j+\frac{\mu_{j-1}}{\mu_j}
 =-\frac{\mu_{j+1}}{\mu_j}.
\]
Taking the reciprocal at $p_T$ gives the stated index on $A$.
Finally, $N_T\cdot A=b_r=1/\mu_{r+1}$, and
\eqref{eq:mu-increasing} gives the coprimality assertion.
\end{proof}

\section{Intersections with invariant curves}\label{sec:invariant}

We first work with a relatively minimal foliation $\F$ on a
smooth projective surface $S$. Assume that $K_{\F}$ is pseudo-effective,
and write $K_{\F}=P+N$ for its Zariski decomposition. No assumption is made on
$p_g(\F)$ or on the Kodaira dimension of $\F$.
We prove \autoref{thm:invariant-gap} and determine when equality holds.

Let $A\not\subset\Supp(N)$ be an invariant irreducible curve, and let
\[
  T_i=C_{i,1}+\cdots+C_{i,r_i},
  \qquad 1\le i\le t,
\]
be the maximal $\F$-chains meeting $A$.
For each $T_i$, denote the sequence \eqref{eq:mu} by $\mu_{i,j}$; thus
$\det M_{T_i}=\mu_{i,r_i+1}$. Set
\[
  \nu_A=Z(\F,A)-t\in\ZZ_{\ge0}.
\]

\begin{lemma}\label{lem:invariant-formula}
One has
\begin{equation}\label{eq:orbifold}
  P\cdot A
  =2p_a(A)-2+\nu_A
   +\sum_{i=1}^t\left(1-\frac1{\mu_{i,r_i+1}}\right).
\end{equation}
If $p_a(A)=\nu_A=0$, \correction{then $A$ is a smooth rational curve,} the points $A\cap T_i$
are the only singularities of $\F$ on $A$, and
\begin{equation}\label{eq:cs-integrality}
  A^2=-\sum_{i=1}^t\frac{\mu_{i,r_i}}{\mu_{i,r_i+1}}\in\ZZ.
\end{equation}
Each fraction in this sum is reduced and lies strictly between zero and one.
\end{lemma}

\begin{proof}
Combine \eqref{eq:indices} with
$N\cdot A=\sum_i1/\mu_{i,r_i+1}$. If $p_a(A)=0$, \correction{then $A$ is a smooth
rational curve.} If also $\nu_A=0$, positivity of the local $Z$-index on a smooth
invariant branch excludes any further foliation singularity. Formula
\eqref{eq:cs-integrality} and the assertion about its fractions follow from
\eqref{eq:indices} and \autoref{lem:terminal}.
\end{proof}

\begin{proposition}\label{prop:invariant-intersection}
Let $\F$ be a relatively minimal foliation on a smooth projective surface
$S$, with $K_{\F}$ pseudo-effective and Zariski decomposition
$K_{\F}=P+N$. Every invariant irreducible curve $A\subset S$ satisfies
\[
 P\cdot A=0\qquad\text{or}\qquad P\cdot A\ge\frac16.
\]
Equality $P\cdot A=1/6$ holds precisely when
$A\not\subset\Supp(N)$ and
\[
 \mathcorrection{t=2},\qquad p_a(A)=0,\qquad \nu_A=1,\qquad
 (\mu_{1,r_1+1},\mu_{2,r_2+1})=(2,3),
\]
after reordering the two chains.
\end{proposition}
\begin{proof}
If $A\subset\Supp(N)$, then $P\cdot A=0$ by the Zariski decomposition.
Otherwise $P\cdot A\ge0$ by nefness. Assume $0<P\cdot A\le1/6$. Formula~\eqref{eq:orbifold} excludes
$p_a(A)\ge1$, since every positive value in that case is at least $1/2$.
\correction{Thus $A$ is a smooth rational curve.} If $\nu_A\ge2$, again every positive value
is at least $1/2$. For $\nu_A=1$, only $t=2$ is possible, and
\[
 0<1-\frac1{\mu_{1,r_1+1}}-\frac1{\mu_{2,r_2+1}}\le\frac16
\]
forces the two determinants to be $2,3$, giving $P\cdot A=1/6$.

Suppose $\nu_A=0$. If $t\le2$, then $P\cdot A<0$, whereas
$t\ge5$ implies $P\cdot A\ge1/2$. If $t=4$, the only tuple giving a value in $(0,1/6]$
is $(2,2,2,3)$. It contradicts \eqref{eq:cs-integrality}, since the first
three fractions sum to $3/2$ and the last is $1/3$ or $2/3$.

For $t=3$, order the chains by increasing determinant. Then
\begin{equation}\label{eq:reciprocal-window}
 \frac56\le
 \frac1{\mu_{1,r_1+1}}+\frac1{\mu_{2,r_2+1}}+\frac1{\mu_{3,r_3+1}}
 <1.
\end{equation}
The first determinant is at most three. If it is two, the second lies
between three and six; if it is three, the second is three or four.
By \eqref{eq:cs-integrality}, the third fraction is congruent modulo
$\ZZ$ to the negative of the sum of the first two. Since all three
fractions are reduced, this determines the third denominator from the
first two fractions. The possibilities are
\[
\begin{array}{c|cccccc}
(\mu_{1,r_1+1},\mu_{2,r_2+1})&(2,3)&(2,4)&(2,5)&(2,6)&(3,3)&(3,4)\\
\mu_{3,r_3+1}&6&4&10&3&3&12
\end{array}
\]
and none satisfies both the ordering and \eqref{eq:reciprocal-window}.
For $(3,3)$, the first two fractions may also sum to an integer.
Adding a third fraction strictly between zero and one would then
contradict \eqref{eq:cs-integrality}.
This excludes $\nu_A=0$ and proves both assertions.
\end{proof}

\begin{proof}[Proof of \autoref{thm:invariant-gap}]
Let $\pi:(S,\F)\to(S_{\min},\F_{\min})$ be a birational morphism
to a relatively minimal model.
\correction{A blow-up between reduced models} has discrepancy zero at a
singular point and one at a regular point. Hence
\[
 K_{\F}=\pi^*K_{\F_{\min}}+E,
\]
where $E$ is effective and $\pi$-exceptional. The divisor
$K_{\F_{\min}}=\pi_*K_{\F}$ is pseudo-effective. Write its Zariski
decomposition as $P_{\min}+N_{\min}$. The divisor $\pi^*P_{\min}$ is nef
and has intersection zero with every component of $\pi^*N_{\min}+E$,
whose support is negative definite.
Uniqueness of the Zariski decomposition therefore gives
\[
 P=\pi^*P_{\min}.
\]
If $A$ is $\pi$-exceptional, then $P\cdot A=0$. Otherwise $\pi_*A$
is an invariant irreducible curve and the projection formula gives
$P\cdot A=P_{\min}\cdot\pi_*A$. The inequality now follows from
\autoref{prop:invariant-intersection}.
The invariant curve in \autoref{ex:invariant-canonical} has $P\cdot A=1/6$,
so the constant is optimal.
\end{proof}

\begin{remark}\label{rem:noninvariant-no-gap}
There is no uniform positive lower bound for the positive intersection
numbers of non-invariant curves with $P$, even when
$p_g(\F)=1$ and $\F$ is relatively minimal of general type.
In \autoref{thm:plane-family}, the last exceptional curve $f_n$
over $U$ is non-invariant and satisfies
\[
 P\cdot f_n=\frac1{n+1}>0,\qquad
 \lim_{n\to\infty}P\cdot f_n=0.
\]
As shown there, this curve survives on a relatively minimal model,
and the projection formula preserves its intersection with the nef part.
\end{remark}

\Needspace{8\baselineskip}
\section{Proof of the lower bound}\label{sec:lower-bounds}

We work on a relatively minimal model of a foliation of general type.
The hypothesis $p_g(\F)>0$ gives an effective divisor $D\in|K_{\F}|$.
We choose $K_{\F}=D$ as the representative in the Zariski decomposition.
When $p_g(\F)=1$, this effective canonical divisor is unique.

\subsection{The effective canonical divisor}\label{sec:core}

\begin{lemma}\label{lem:effective-nef}
One has $\Supp(N)\subset\Supp(D)$ and $P=D-N\ge0$.
\end{lemma}

\begin{proof}
Fix a maximal chain.  Since $D\cdot C_1=-1$, the curve $C_1$ must occur in
$D$.  If $C_{j-1}$ occurs in $D$ but $C_j$ does not, then the contribution
of $C_{j-1}$ gives $D\cdot C_j>0$, contradicting
$D\cdot C_j=0$.  Hence every $C_j$ occurs in $D$.  Its coefficient in $D$
is a positive integer, while its coefficient in $N$ is smaller than one;
therefore $D-N\ge0$.
\end{proof}

Write
\begin{equation}\label{eq:core}
  D=D_N+D',
  \qquad
  D'=\sum_A m_AA,
\end{equation}
where $D_N$ is supported on $N$, $D'$ has no component in $\Supp(N)$, and
$m_A\in\ZZ_{>0}$.

The Zariski decomposition gives $P\cdot N=P\cdot D_N=0$, hence
\begin{equation}\label{eq:core-volume}
  P^2=P\cdot D'
     =\sum_{A\subset\Supp(D')}m_A(P\cdot A).
\end{equation}
Every summand is non-negative by nefness.

For a maximal chain $T=C_1+\cdots+C_r$, we use the matrix and sequences
of \autoref{sec:chains}.
Let $a_j$ and $b_j$ be the coefficients of $D$ and $N$ along $C_j$, and set
\[
  s_j=D'\cdot C_j.
\]
Write $\mathbf a=(a_j)^t$, $\mathbf b=(b_j)^t$, and
$\mathbf s=(s_j)^t$.

\begin{lemma}\label{lem:chain-equations}
Let $\mathbf e_j$ denote the $j$-th standard basis vector of $\mathbb R^r$.
Then
\begin{equation}\label{eq:linear-chain}
  M_T\mathbf a=\mathbf e_1+\mathbf s,
  \qquad
  M_T\mathbf b=\mathbf e_1,
  \qquad
  \mathbf a-\mathbf b=M_T^{-1}\mathbf s.
\end{equation}
The coefficient $a_r$ of $C_r$ in $D$ satisfies
\begin{equation}\label{eq:integrality}
 a_r=\frac{1+\sum_{j=1}^r s_j\mu_j}{\mu_{r+1}}\in\ZZ_{>0}.
\end{equation}
Consequently,
\begin{equation}\label{eq:det-bound}
 \mu_{r+1}\mid\left(1+\sum_j s_j\mu_j\right),\qquad
 \mu_{r+1}\le1+\sum_j s_j\mu_j.
\end{equation}
\end{lemma}
\begin{proof}
With $a_0=a_{r+1}=0$, equation \eqref{eq:chain-int} gives
\[
  -e_ja_j+a_{j-1}+a_{j+1}+s_j=-\delta_{j1}.
\]
These equalities give $M_T\mathbf a=\mathbf e_1+\mathbf s$. Since
$P\cdot C_j=0$, one also has
$N\cdot C_j=D\cdot C_j=-\delta_{j1}$, giving
$M_T\mathbf b=\mathbf e_1$. Subtraction proves the last equality
in \eqref{eq:linear-chain}.

The last row of $\mathbf a=M_T^{-1}(\mathbf e_1+\mathbf s)$,
given by \eqref{eq:inverse-matrix}, yields \eqref{eq:integrality}.
The coefficient $a_r$ is a positive integer because $D$ is an effective
integral divisor containing $T$. This proves \eqref{eq:det-bound}.
\end{proof}

\subsection{Non-invariant components}

Let $A\subset\Supp(D')$ be non-invariant, of multiplicity $m$ in $D$, and
write $D=mA+R$. From
\[
  K_{\F}\cdot A=mA^2+R\cdot A,
  \qquad
  \tang(\F,A)=K_{\F}\cdot A+A^2,
\]
we obtain
\begin{equation}\label{eq:transverse-K}
  K_{\F}\cdot A=\frac{m\,\tang(\F,A)+R\cdot A}{m+1}.
\end{equation}

\begin{lemma}\label{lem:boundary}
If $A\cdot C_j=k>0$, then
\begin{equation}\label{eq:boundary-positive}
  \frac{a_j}{m+1}-b_j\ge0.
\end{equation}
Equality implies $j=1$, $k=1$, and $\mathbf s=m\mathbf e_1$.
\end{lemma}

\begin{proof}
Since $mA$ is part of $D'$, one has $s_j\ge mk$. By
\eqref{eq:linear-chain}, positivity of $M_T^{-1}$, and
\eqref{eq:b-formulas},
\[
  a_j-b_j
  \ge(M_T^{-1})_{jj}s_j
  \ge mk\mu_jb_j
  \ge mb_j.
\]
This proves \eqref{eq:boundary-positive}. Equality requires $k=1$,
$\mu_j=1$, $s_j=m$, and $s_i=0$ for $i\ne j$. Hence $j=1$ and
$\mathbf s=m\mathbf e_1$.
\end{proof}

\begin{proposition}\label{prop:tangency-decomposition}
For every non-invariant component $A$ of multiplicity $m$ in $D'$,
\begin{equation}\label{eq:transverse-decomposition}
\begin{split}
  m(P\cdot A)={}&\frac{m^2}{m+1}\tang(\F,A)
  +\sum_{A'\ne A}\frac{mm_{A'}}{m+1}(A'\cdot A)\\
  &+\sum_T\sum_{j=1}^{r(T)}
  m(A\cdot C_{T,j})
  \left(\frac{a_{T,j}}{m+1}-b_{T,j}\right).
\end{split}
\end{equation}
Every term on the right is non-negative.
Here the first sum ranges over all components $A'\ne A$ of $D'$, and the
second over all maximal chains in $\Supp(N)$.
\end{proposition}

\begin{proof}
Since $P=K_{\F}-N$, equation \eqref{eq:transverse-K} gives
\[
  m(P\cdot A)
  =\frac{m^2}{m+1}\tang(\F,A)
   +\frac{m}{m+1}(R\cdot A)-m(N\cdot A).
\]
Now expand
\[
  R=\sum_{A'\ne A}m_{A'}A'+\sum_T\sum_ja_{T,j}C_{T,j},
  \qquad
  N=\sum_T\sum_jb_{T,j}C_{T,j}.
\]
Collecting the two terms involving each $C_{T,j}$ gives
\eqref{eq:transverse-decomposition}.  Its last terms are non-negative by
\autoref{lem:boundary}; the remaining terms are non-negative by
\eqref{eq:tangency} and the effectivity of $D'$.
\end{proof}

\Needspace{12\baselineskip}
\begin{corollary}\label{cor:small-components}
Suppose that
\begin{equation}\label{eq:small-assumption}
  0<P^2\le\frac16.
\end{equation}
Then an invariant component $A$ of $D'$ has $P\cdot A=0$ or $1/6$.
In the latter case, $P^2=1/6$ and its coefficient in $D$ is one.
Every non-invariant component $A$ is smooth, avoids $\Sing(\F)$, and is
everywhere transverse to $\F$. Moreover,
\begin{equation}\label{eq:disjoint-core}
  \tang(\F,A)=0,
  \qquad
  A\cdot A'=0
  \quad(A'\ne A,\ A'\subset\Supp(D')).
\end{equation}
\end{corollary}

\begin{proof}
By \eqref{eq:core-volume}, $m_A(P\cdot A)\le1/6$ for every component
$A$ of $D'$. The assertion for invariant components follows from
\autoref{prop:invariant-intersection}. For a non-invariant component $A$ of multiplicity $m$, if
$\tang(\F,A)>0$, its contribution to
\eqref{eq:transverse-decomposition} is at least $m^2/(m+1)\ge1/2$.
If $A\cdot A'>0$ for some other component $A'$ of $D'$, the corresponding
term is at least $mm_{A'}/(m+1)\ge1/2$.
Thus \eqref{eq:disjoint-core} holds.

If $f=0$ is a reduced local equation of $A$ and $v$ a local saturated
vector field defining $\F$, the local tangency number at $x\in A$ is
\[
 \dim_{\CC}\OO_{S,x}/(f,v(f)).
\]
It is positive whenever $A$ is singular at $x$, the vector field vanishes
at $x$, or $A$ is tangent to $\F$ at $x$. Since its sum is zero, none of
these occurs. The chain components are invariant, so their intersections
with $A$ are transverse as well.
\end{proof}

\subsection{Chains meeting invariant and non-invariant components}

For the remainder of the proof, assume \eqref{eq:small-assumption}.
We use McQuillan's classification of curves $C$ with $P\cdot C=0$
and their invariance \cite[III.3.2 and Remark~III.2.2]{McQ08}; see also
\cite[Proposition~2.15 and Theorem~2.16]{CF18}.

\begin{lemma}[McQuillan]\label{lem:zero-types}
Let $\F$ be a relatively minimal foliation of general type on a smooth
surface $S$, and let $K_{\F}=P+N$ be its Zariski decomposition.
Let $Z$ be the union of all curves $A$ such that $P\cdot A=0$.
Then $Z$ is the union of
\begin{enumerate}[label=\textup{(\arabic*)}]
  \item the support of $N$;
  \item disjoint chains of rational curves, none of which is contained
    in $\Supp(N)$;
  \item cycles $\Gamma$ of rational curves such that
    $\Sing(\F)\cap\Gamma=\Sing(\Gamma)$;
  \item single rational nodal curves $A$ such that
    $\Sing(\F)\cap A=\Sing(A)$.
\end{enumerate}
Moreover, a chain $\mathcal C$ of type~\textup{(2)} is either disjoint from
$\Supp(N)$ or meets exactly two connected components of $\Supp(N)$.
Each of these consists of a smooth rational curve $E_i$ with
$E_i^2=-2$, $i=1,2$, and both $E_1$ and $E_2$ meet $\mathcal C$
transversely along the same end component $A$, at distinct points
$p_1,p_2$, with
\begin{equation*}
  N|_A=\frac12p_1+\frac12p_2.
\end{equation*}
Every irreducible component of $Z$ is $\F$-invariant. The configurations
of types~\textup{(3)} and~\textup{(4)} meet no invariant curve outside
themselves; in particular, they are disjoint from $\Supp(N)$.
\end{lemma}

\begin{lemma}\label{lem:mixed-cost}
Suppose an invariant component $A$ of $D'$, with coefficient $c_A$,
meets a maximal chain $T=C_1+\cdots+C_r$. If a non-invariant component of
coefficient $m$ meets $C_j$, its contribution at each such point to
\eqref{eq:transverse-decomposition} is at least
\begin{equation*}
  \frac{c_Am\mu_j}{\mu_{r+1}(m+1)}
  \ge\frac{c_A}{2\mu_{r+1}}.
\end{equation*}
\end{lemma}

\begin{proof}
By \autoref{lem:terminal}, $A$ meets $T$ only at $C_r$, so
$\mathbf s\ge c_A\mathbf e_r+m\mathbf e_j$. Using
\eqref{eq:linear-chain} and \eqref{eq:b-formulas}, we obtain
\[
  a_j-b_j\ge\frac{c_A\mu_j}{\mu_{r+1}}+m\mu_jb_j.
\]
Since $\mu_j\ge1$ and $m/(m+1)\ge1/2$, we obtain
\[
  m\left(\frac{a_j}{m+1}-b_j\right)
  \ge\frac{m}{m+1}
     \left(\frac{c_A\mu_j}{\mu_{r+1}}+m(\mu_j-1)b_j\right)
  \ge\frac{c_A}{2\mu_{r+1}}.\qedhere
\]
\end{proof}

\begin{lemma}\label{lem:no-mixed}
Under \eqref{eq:small-assumption}, no maximal chain meets both an
invariant and a non-invariant component of $D'$.
\end{lemma}
\begin{proof}
Suppose a chain $T$ meets an invariant component $A$ and a non-invariant
component of $D'$. If $P\cdot A>0$, \autoref{prop:invariant-intersection}
gives an invariant contribution of at least $1/6$, while
\autoref{lem:mixed-cost} gives an additional positive contribution.
This contradicts $P^2\le1/6$.

If $P\cdot A=0$, \autoref{lem:zero-types} implies that $A$ is contained
in a chain of
type~\textup{(2)}, since it lies outside $\Supp(N)$ and meets $T$.
Thus $T$ is a single $(-2)$-curve. \autoref{lem:mixed-cost}, with
$\mu_2=2$, gives a contribution of at least $c_A/4>1/6$.
Again this contradicts \eqref{eq:small-assumption}.
\end{proof}

\subsection{The contribution of a maximal chain}

Continue to assume \eqref{eq:small-assumption}. Let the maximal chain
$T=C_1+\cdots+C_r$ meet non-invariant components of $D'$. By
\autoref{lem:no-mixed}, it meets no invariant component of $D'$.
Index its intersection points with $D'$ by $\alpha$.  Let
$p_\alpha\in C_{j_\alpha}$ and let $m_\alpha$ be the coefficient in $D$ of
the component passing through $p_\alpha$. Distinct components of $D'$
are disjoint at these points, and repeated intersections of the same
component are counted separately. Define
\begin{equation*}
  \Psi(T)=\sum_\alpha
  m_\alpha\left(
    \frac{a_{j_\alpha}}{m_\alpha+1}-b_{j_\alpha}
  \right).
\end{equation*}
Every summand is non-negative.  Put
\begin{equation*}
  s_j=\sum_{j_\alpha=j}m_\alpha.
\end{equation*}
Transversality and the absence of invariant components meeting $T$ imply
$s_j=D'\cdot C_j$. Thus \eqref{eq:linear-chain} applies with these $s_j$.

\begin{proposition}\label{prop:pure-chain}
Suppose every non-invariant component of $D'$ is smooth, transverse
to $\F$, and disjoint from every other component of $D'$. Let $T$
be a maximal chain meeting $D'$ but disjoint from its invariant
components. Then
\[
  \Psi(T)=0
  \qquad\text{or}\qquad
  \Psi(T)\ge\frac16.
\]
The equality $\Psi(T)=0$ holds precisely when $T\cap\Supp(D')$
is one point on $C_1$. The equality $\Psi(T)=1/6$ holds precisely
when $T=[-2,-2]$ and $T\cap\Supp(D')$ is one point on $C_2$,
lying on a component of coefficient one in $D$.
\end{proposition}

\begin{proof}
For every $j$ with $s_j>0$, the inequality
$m_\alpha\le s_j$ for $j_\alpha=j$ gives
\[
 \sum_{j_\alpha=j}\frac{m_\alpha}{m_\alpha+1}
 \ge\frac{\sum_{j_\alpha=j}m_\alpha}{s_j+1}
 =\frac{s_j}{s_j+1}.
\]
Multiplying by $a_j$ and subtracting $s_jb_j$, we obtain
\begin{equation}\label{eq:grouping}
 \sum_{j_\alpha=j}
 m_\alpha\left(\frac{a_j}{m_\alpha+1}-b_j\right)
 \ge\frac{s_j}{s_j+1}\bigl(a_j-(s_j+1)b_j\bigr).
\end{equation}

Let
\[
  J=\max\{j:s_j>0\}.
\]
Thus $s_k=0$ for $k>J$.  Taking the $J$-th row of
\eqref{eq:linear-chain} and using \eqref{eq:inverse-matrix}, we obtain
\begin{equation}\label{eq:last-contact}
  a_J-b_J
  =\frac{\xi_J}{\mu_{r+1}}\sum_{k=1}^Js_k\mu_k
  =b_J\sum_{k=1}^Js_k\mu_k.
\end{equation}
Taking the last row of
$\mathbf a=M_T^{-1}(\mathbf e_1+\mathbf s)$ gives
\begin{equation}\label{eq:ar-final}
  a_r=\frac{1+\sum_{k=1}^Js_k\mu_k}{\mu_{r+1}}
  \in\ZZ_{>0}.
\end{equation}
Hence
\begin{equation}\label{eq:bJ-final}
  b_J=\frac{\xi_J}{\mu_{r+1}}
  \ge\frac1{1+\sum_{k=1}^Js_k\mu_k}.
\end{equation}

Assume first that $J>1$.  Keeping only the non-negative terms supported on
$C_J$, and using \eqref{eq:grouping}, we find
\begin{align*}
  \Psi(T)
  &\ge\frac{s_J}{s_J+1}
    \bigl(a_J-(s_J+1)b_J\bigr)\\
  &=\frac{s_Jb_J}{s_J+1}
    \left(\sum_{k=1}^Js_k\mu_k-s_J\right)\\
  &\ge
    \frac{s_J\left(\sum_{k=1}^Js_k\mu_k-s_J\right)}
    {(s_J+1)\left(1+\sum_{k=1}^Js_k\mu_k\right)}.
\end{align*}
Since $J>1$, equation \eqref{eq:mu-increasing} gives $\mu_J\ge2$, and
therefore
\[
  \sum_{k=1}^Js_k\mu_k\ge2s_J.
\]
Since $(x-s_J)/(x+1)$ is increasing for $x\ge2s_J$, it follows that
\[
  \Psi(T)
  \ge\frac{s_J^2}{(s_J+1)(2s_J+1)}
  \ge\frac16,
\]
because
\[
  \frac{s_J^2}{(s_J+1)(2s_J+1)}-\frac16
  =\frac{(s_J-1)(4s_J+1)}
  {6(s_J+1)(2s_J+1)}\ge0.
\]

Equality $\Psi(T)=1/6$ forces $s_J=1$ and
$\sum_{k=1}^J s_k\mu_k=2$, since the preceding inequalities are strict
otherwise. Hence $\mu_J=2$, so $J=2$, and $s_k=0$ for $k<J$.
By \eqref{eq:ar-final}, $\mu_{r+1}$ divides three. Since
$\mu_{r+1}>\mu_J=2$, it equals three. Thus $r=2$, $e_1=e_2=2$,
and the unique intersection point lies on a component of coefficient
one in $D$. Conversely, the inverse
matrix of $[-2,-2]$ gives $\Psi(T)=1/6$ for this configuration.

Assume next that $J=1$.  Put
$s=s_1=\sum_\alpha m_\alpha$.  Equations
\eqref{eq:last-contact} and \eqref{eq:bJ-final} give
\[
  a_1=(s+1)b_1,
  \qquad
  b_1\ge\frac1{s+1}.
\]
Consequently,
\begin{equation*}
  \Psi(T)=b_1\sum_\alpha
  \frac{m_\alpha(s-m_\alpha)}{m_\alpha+1}.
\end{equation*}
If there is one intersection point, then $s=m_\alpha$ and $\Psi(T)=0$.
If there are $q\ge2$ points, then $s\ge q\ge2$ and
$m_\alpha/(m_\alpha+1)\ge1/2$.  Thus
\[
  \Psi(T)
  \ge\frac1{2(s+1)}\sum_\alpha(s-m_\alpha)
  =\frac{(q-1)s}{2(s+1)}
  \ge\frac13.
\]
\end{proof}

\begin{proof}[Proof of \autoref{thm:lower-bound}]
Suppose $0<P^2<1/6$. By \autoref{cor:small-components}, all
invariant components $A$ of $D'$ satisfy $P\cdot A=0$, while the
non-invariant components have zero tangency and are disjoint from the
other components of $D'$. By \autoref{lem:no-mixed}, every maximal
chain meeting a non-invariant component of $D'$ is disjoint from its
invariant components.
Equations~\eqref{eq:core-volume} and \eqref{eq:transverse-decomposition}
therefore give
\[
 P^2=\sum_T\Psi(T),
\]
where $T$ runs over the maximal $\F$-chains meeting a non-invariant
component of $D'$. By \autoref{prop:pure-chain}, every summand
is zero or at least
$1/6$. Their sum is smaller than $1/6$, so all vanish, contradicting
$P^2>0$. Hence $\vol(\F)\ge1/6$. Sharpness follows from
\autoref{thm:plane-family} and \autoref{ex:invariant-canonical}.
\end{proof}

\section{The equality case}\label{sec:equality}

Assume throughout this section that $p_g(\F)>0$ and $P^2=1/6$ on a relatively minimal model. By \eqref{eq:LT-intro}, $p_g(\F)=1$, so $D\in|K_{\F}|$ is unique.

We first make two observations. Every maximal $\F$-chain meets $D'$.
Otherwise \eqref{eq:linear-chain} would give $a_r=1/\mu_{r+1}<1$,
contrary to $a_r\in\ZZ_{>0}$. Also, an effective nef and big divisor has
positive intersection with each connected component of its support.
Indeed, write $P=P_1+P_2$, where $\Supp(P_1)$ is one connected component
of $\Supp(P)$ and $\Supp(P_1)\cap\Supp(P_2)=\varnothing$.
If $P\cdot P_1=0$, then $P_1^2=0$. The Hodge index theorem implies that
$P_1$ is numerically trivial, which contradicts its effectivity.

\begin{theorem}\label{thm:equality-core}
Let $(S,\F)$ be a relatively minimal foliated surface with $p_g(\F)>0$ and $\vol(\F)=1/6$. Let $D\in|K_{\F}|$, and write $D=D_N+D'$ as in \eqref{eq:core}. Then exactly one of the following holds.
\begin{enumerate}
\item Every component of $D$ is $\F$-invariant. There is a unique component $A$ of $D'$ such that $P\cdot A>0$.
Its coefficient in $D$ is one, and
\[
 P\cdot A=\frac16,\qquad
 \mathcorrection{t=2},\quad p_a(A)=0,\quad \nu_A=1,\quad
(\mu_{1,r_1+1},\mu_{2,r_2+1})=(2,3).
\]
Every other component $C$ of $D'$ satisfies $P\cdot C=0$.
Every component of $D$ has rational normalization, and $q(S)=0$.
\item The divisor $D'$ is a single smooth curve $A$, with coefficient one, transverse to $\F$. There are a chain $T=C_1+C_2=[-2,-2]$ and $k\ge1$ mutually disjoint $(-2)$-curves $E_1,\ldots,E_k$, where $k$ is odd, such that
\begin{align}\label{eq:transverse-equality}
 D&=A+C_1+C_2+\sum_{i=1}^k E_i,\nonumber\\
 N&=\frac23C_1+\frac13C_2+\frac12\sum_{i=1}^kE_i,\\
 A^2&=-\frac{k+1}{2},\qquad
 A\cdot C_1=0,\quad A\cdot C_2=A\cdot E_i=1.\nonumber
\end{align}
The intersections involving $A$ are transverse and lie away from $\Sing(\F)$, and $P\cdot A=1/6$.
\end{enumerate}
Both cases occur on rational surfaces.
\end{theorem}
\begin{proof}
By \autoref{cor:small-components} and \autoref{lem:no-mixed},
$P^2$ is the sum of the terms $m_A(P\cdot A)$ for invariant components
$A$ of $D'$ and the terms $\Psi(T)$ for maximal chains meeting
non-invariant components of $D'$. Each positive term is at least
$1/6$, by \autoref{prop:invariant-intersection} and
\autoref{prop:pure-chain}. Thus exactly one term is positive.

Suppose $m_A(P\cdot A)>0$ for an invariant component $A$ of $D'$.
\autoref{prop:invariant-intersection} gives the stated data for $A$.
Every other component $C$ of $D'$ satisfies $P\cdot C=0$ and is
therefore invariant by \autoref{lem:zero-types}. The same classification shows
that these curves have rational normalization. Thus every component of
$D$ is invariant and has rational normalization.
Were the Albanese map nonconstant, the pullback $H$ of an ample divisor on its image would be a non-zero nef class with $D\cdot H=0$, since every rational curve is contracted by the Albanese map. This is impossible for a big divisor $D$: write $D\equiv H_0+E$ with $H_0$ ample and $E$ effective to obtain $D\cdot H>0$. Consequently $q(S)=0$.

Suppose $\Psi(T)>0$ for a maximal chain $T$ meeting a non-invariant
component of $D'$. By \autoref{prop:pure-chain},
$T=C_1+C_2=[-2,-2]$ and it meets $D'$ at one point on $C_2$.
The component $A$ of $D'$ through this point has coefficient one in $D$.
Every other maximal chain $T'$ meeting a non-invariant component of
$D'$ has $\Psi(T')=0$ and meets $D'$ at one point on its first component.
By \autoref{lem:no-mixed} and \eqref{eq:disjoint-core}, all other
components of $D'$, together with the maximal chains meeting them, are
disjoint from the connected component of $\Supp(P)$ containing $A$.
Each such connected component has intersection zero with $P$,
contrary to the observation preceding the theorem. Thus $D'=A$,
with coefficient one.

For any other maximal chain $T'$ meeting $A$, one has $\Psi(T')=0$.
The coefficient of its last component in $D$ is
$2/\mu_{r+1}\in\ZZ_{>0}$. Since $\mu_{r+1}\ge2$, it follows that
$\mu_{r+1}=2$, so $T'$ is a single $(-2)$-curve with coefficient one in
$D$. The components $C_1$ and $C_2$ also have coefficient one in $D$.
This proves the first two lines of \eqref{eq:transverse-equality}.
Finally, $\tang(\F,A)=0$ gives
\[
 0=K_{\F}\cdot A+A^2=2A^2+k+1.
\]
Thus $k$ is odd, in particular positive, and the claimed self-intersection
follows. \autoref{thm:plane-family} with $n=2$ realizes the transverse
case, and \autoref{ex:invariant-canonical} realizes the invariant one.
\autoref{thm:equality-intro} follows as well.
\end{proof}

\begin{remark}
\autoref{thm:equality-core} describes the effective canonical divisor
when $\vol(\F)=1/6$. It does not classify these foliations up to birational
equivalence. If $p_g(\F)>0$ and $q(S)>0$, this divisor must have the
configuration \eqref{eq:transverse-equality}.
\end{remark}

\section{Foliations on irregular surfaces}\label{sec:albanese-geography}

\subsection{The Albanese fibration}\label{sec:alb}

We first show that a foliation with $p_g(\F)=0$ on an irregular surface
is induced by the Albanese fibration. \autoref{cor:alb-intro} then
follows from \autoref{thm:lower-bound}.

\begin{proposition}\label{prop:alb}
Let $(S,\F)$ be a foliated surface with $q(S)>0$. If $p_g(\F)=0$, then the Albanese image of $S$ is a curve and $\F$ is induced by its Albanese fibration $f:S\to B$. In particular, $g(B)=q(S)$.
\end{proposition}
\begin{proof}
Pass first to a reduced model $X$ of $\F$. Both irregularity and the Albanese variety are unchanged. Dualizing $T_{\F}\subset T_X$ gives a morphism
\[
 \Omega_X^1\longrightarrow K_{\F}.
\]
Since $\dim H^0(X,K_{\F})=p_g(\F)=0$, every global holomorphic one-form annihilates $T_{\F}$. Write $a_X:X\to\Alb(X)$ for the Albanese map. Holomorphic one-forms on $X$ are pullbacks of invariant one-forms on $\Alb(X)$, and the latter span its cotangent spaces. Thus $da_X(T_{\F})=0$.
The Albanese map is nonconstant because $q(X)>0$, and the identity $da_X(T_{\F})=0$ shows that its generic rank is at most one. Thus its image is a curve, denoted by $B$.
Namely, the Albanese map gives a fibration
\[
 X\xrightarrow{f_X}B= a_X(X)\subseteq\Alb(X),
\]
where $B$ is a smooth projective curve and $f_X$ has connected fibers. At the generic point, $T_{\F}$ is the kernel of $df_X$.
Therefore, $\F$ is nothing but the foliation induced by $f_X$.

Every exceptional curve of $X\to S$ is contracted by $f_X$, since $g(B)=q(X)=q(S)>0$ by assumption. Thus $f_X$ descends to a morphism $f:S\to B$ with connected fibers, which is the Albanese fibration of $S$. Its induced saturated foliation agrees with the original foliation on a dense open subset, and hence equals it on $S$.
\end{proof}

\begin{proof}[Proof of \autoref{cor:alb-intro}]
The volume assumption and \autoref{thm:lower-bound} give $p_g(\F)=0$.
Hence the conclusion follows from \autoref{prop:alb}.
\end{proof}



\begin{corollary}\label{cor:uniqueness}\label{cor:large-region}
Let $S$ be a smooth projective irregular surface.
\begin{enumerate}
\item There is at most one saturated foliation of general type on $S$
with volume in $(0,1/6)$.
\item If $S$ has maximal Albanese dimension or is not of general type,
every foliation of general type on $S$ has volume at least $1/6$.
The same bound holds for every transcendental foliation of general type
on $S$.
\end{enumerate}
\end{corollary}
\begin{proof}
For (i), \autoref{cor:alb-intro} identifies every such foliation with
the saturation of the kernel of the differential of the same Albanese
fibration. A saturated rank-one subsheaf of $T_S$ is determined by its restriction
to a dense open subset of $S$.
For (ii), suppose the volume is less than $1/6$. By \autoref{cor:alb-intro}, the foliation is induced by the Albanese fibration, so it is algebraically integrable and the Albanese image is a curve. It remains to see that $S$ is of general type. On a reduced model $X$, write $f_X:X\to B$ for the induced fibration. Its canonical divisor satisfies
\[
K_X=K_{\F_X}+f_X^*K_B+\sum_i(m_i-1)C_i,
\]
where $m_i$ is the multiplicity of the fiber component $C_i$. Since $g(B)\ge1$, the last two terms are linearly equivalent to effective divisors. Thus $K_X$ is big because $K_{\F_X}$ is big, and $S$ is of general type as well. This proves all the assertions in (ii).
\end{proof}

\subsection{Denominators and the genus of a fiber}\label{sec:denominators}

We next bound the denominators of the coefficients of $N$ for foliations
induced by fibrations, without assuming $p_g(\F)>0$. Let $\F$ be induced
by a fibration of genus $g\ge2$, and work on a relatively minimal model
$(X,\F)$. The first integral is a morphism $f:X\to B$: a reduced
algebraically integrable foliation has no dicritical singularities, and
every $\F$-exceptional curve is contained in a fiber. Write $K_{\F}=P+N$
as before. Every component of $N$ is contained in a fiber, so for a
general fiber $F$,
\begin{equation}\label{eq:fibre-degree}
 P\cdot F=K_{\F}\cdot F=2g-2.
\end{equation}

For $g\ge2$, set
\begin{equation*}
 \ell_g=\lcm\{1,2,\ldots,4g+2\}.
\end{equation*}

\begin{proposition}\label{prop:denominator}
Let $\F$ be induced by a fibration of genus $g\ge2$. On a relatively
minimal model, every maximal $\F$-chain has determinant at most $4g+2$.
Moreover, $\ell_gP$ and $\ell_gN$ are integral divisors, and
\[
 \ell_g\vol(\F)\in\ZZ_{\ge0}.
\]
\correction{The determinant bound is also sharp for foliations of general type.}
If the volume is positive, it is at least $1/\ell_g$.
\end{proposition}

Let $i(\F)$ denote the least positive integer such that $i(\F)N$
is integral, equivalently such that $i(\F)P$ is integral.
Pereira--Svaldi \cite[Proposition~5.3]{PS19} prove
$i(\F)\le(42(2g-2))!$. We use the next two lemmas to prove the
stronger divisibility statement
\[
 i(\F)\mid\ell_g.
\]
Han--Jiao--Li--Liu \cite{HJLL24} prove a more general discreteness theorem
for foliations with fixed volume of a general leaf. The constant $4g+2$
is Wiman's classical bound \cite{Wim95}; the corresponding bound for local
monodromy is proved in \cite[Lemma~3.3]{LiuTan17}. The example at the end
of the proof shows that the determinant bound is sharp. We do not claim
that the volume bound $1/\ell_g$ is sharp.

\begin{lemma}\label{lem:tail-multiplicity}
Let $T=C_1+\cdots+C_r$ be a maximal $\F$-chain. The other separatrix
through $p_T$ belongs to a unique fiber component $A$ outside $\Supp(N)$.
If $m_A$ is its multiplicity in the fiber, then $\mu_{r+1}\mid m_A$.
Moreover, either $P\cdot A>0$ or $\mu_{r+1}=2$.
\end{lemma}
\begin{proof}
The other separatrix is algebraic because the foliation is induced by
$f$. It cannot belong to another maximal chain, since the maximal
chains in $\Supp(N)$ are disjoint. Thus it is a branch of a component
$A$ outside $\Supp(N)$.

Choose local coordinates at $p_T$ in which $C_r=(x=0)$ and $A=(y=0)$.
If their fiber multiplicities are $a$ and $m_A$, a local equation of the
fibration is a unit times $x^a y^{m_A}$. Comparing the Camacho--Sad
index along $A$ with \autoref{lem:terminal} gives
\[
 \frac{a}{m_A}=\frac{\mu_r}{\mu_{r+1}},\qquad
 \gcd(\mu_r,\mu_{r+1})=1.
\]
Hence $\mu_{r+1}\mid m_A$.

Suppose $P\cdot A=0$. Since $A$ meets a maximal chain,
\eqref{eq:orbifold} forces $p_a(A)=0$ and $\nu_A\le1$.
If $\nu_A=1$, the same formula implies that $A$ meets exactly two maximal
chains, both of determinant two. If $\nu_A=0$, the singularities on $A$
are precisely its intersections with the maximal chains meeting it. At each such point
the two separatrices are $A$ and the last component of the corresponding
chain. At each node of a chain, the two separatrices are its two
components through that point. Thus no other fiber component meets
their union. Connectedness forces the whole fiber to be supported on
$A$ and these chains. Thus $P$ has intersection zero with the whole
fiber, contrary to \eqref{eq:fibre-degree}. This proves the assertion.
\end{proof}

\begin{lemma}\label{lem:lcm-tail}
Let $A\not\subset\Supp(N)$ be a fiber component meeting $\Supp(N)$, and let
$T_i=C_{i,1}+\cdots+C_{i,r_i}$, $1\le i\le t$, be the maximal
$\F$-chains meeting $A$. Put
\[
 M=\lcm(\mu_{1,r_1+1},\ldots,\mu_{t,r_t+1}).
\]
Then $M\le4g+2$.
\end{lemma}
\begin{proof}
By \autoref{lem:tail-multiplicity}, $M\mid m_A$.
If $P\cdot A=0$, all the determinants are two, so the assertion holds.
Otherwise
\begin{equation}\label{eq:tail-budget}
 M(P\cdot A)\le m_A(P\cdot A)\le2g-2.
\end{equation}
If $P\cdot A\ge1/2$, this gives $M\le4g-4$.

Suppose $0<P\cdot A<1/2$. Formula~\eqref{eq:orbifold} forces $p_a(A)=0$.
If $\nu_A\ge1$, the only case is $\nu_A=1$, $t=2$. We may assume
$\mu_{1,r_1+1}\le\mu_{2,r_2+1}$. The elementary inequality
\begin{equation*}
 M\left(1-\frac1{\mu_{1,r_1+1}}-\frac1{\mu_{2,r_2+1}}\right)
 \ge\frac M2-2
\end{equation*}
gives $M\le4g$. To check it, if the first determinant is two, use
$M/\mu_{2,r_2+1}\le2$. If the first is at least four, or if it is
three and the second is at least six, use that their reciprocal sum is
at most $1/2$. \correction{The inequality can be verified directly for the remaining pairs $(3,3),(3,4),(3,5)$.}

It remains to consider $\nu_A=0$. All singularities on the smooth
rational curve $A$ are intersections with the chains, so
\eqref{eq:cs-integrality} applies. Choose a coordinate $z$ on
$A\simeq\PP^1$ such that these points have finite coordinates
$z_1,\ldots,z_t$. The equation
\begin{equation*}
 w^M=\prod_{i=1}^t(z-z_i)^{M\mu_{i,r_i}/\mu_{i,r_i+1}}
\end{equation*}
defines, after normalization, a connected cyclic cover $C\to\PP^1$
of degree $M$.
Indeed, for each prime $p\mid M$, choose a determinant
$\mu_{i,r_i+1}$ with the same $p$-adic valuation as $M$.
The corresponding exponent is not divisible by $p$, since
$\gcd(\mu_{i,r_i},\mu_{i,r_i+1})=1$.
There is no ramification at infinity by \eqref{eq:cs-integrality}.
The ramification indices are $\mu_{i,r_i+1}$, so Riemann--Hurwitz gives
\[
 2g(C)-2
 =M\left(-2+\sum_i\left(1-\frac1{\mu_{i,r_i+1}}\right)\right)
 =M(P\cdot A).
\]
Positivity and \eqref{eq:tail-budget} imply $2\le g(C)\le g$.
Wiman's bound \cite{Wim95} for the order of an automorphism of a
curve of genus at least two therefore gives $M\le4g(C)+2\le4g+2$.
\end{proof}

\begin{proof}[Proof of \autoref{prop:denominator}]
Let $T=C_1+\cdots+C_r$ be a maximal $\F$-chain. By
\autoref{lem:tail-multiplicity}, its last component meets a fiber
component $A$ outside $\Supp(N)$. Let $M$ be the least common multiple
of the determinants of all maximal chains meeting $A$. Since
$\det M_T=\mu_{r+1}$ divides $M$, \autoref{lem:lcm-tail} gives
\begin{equation}\label{eq:chain-det-bound}
 \det M_T=\mu_{r+1}\le M\le4g+2.
\end{equation}
In particular, $\det M_T$ divides $\ell_g$.

By \autoref{cor:negative-coefficients}, every coefficient of $N_T$
has denominator dividing $\det M_T$, and the coefficient of $C_r$ is
$1/\det M_T$. Thus, for a positive integer $m$, the divisor $mN_T$
is integral if and only if $\det M_T$ divides $m$. Since the maximal
chains are disjoint, it follows that
\[
 i(\F)=\lcm_T\{\det M_T\}\mid\ell_g,
\]
where the least common multiple is $1$ when $N=0$.
Consequently, $\ell_gN$ is integral, and so is
$\ell_gP=\ell_gK_{\F}-\ell_gN$. The equality $P\cdot N=0$ now gives
\begin{equation*}
 \ell_g\vol(\F)=\ell_gP^2=(\ell_gP)\cdot K_{\F}\in\ZZ_{\ge0}.
\end{equation*}
For a positive volume this integer is at least one.

For sharpness of \eqref{eq:chain-det-bound}, the genus-$g$ fibration inducing a foliation of general type in L\"u--Tan \cite[Example~9.4]{LT24} has in its negative part a maximal chain of $4g+1$ consecutive $(-2)$-curves, with coefficients $(4g+2-i)/(4g+2)$. Its determinant is $4g+2$.
\end{proof}

If the genus $g$ of a general leaf satisfies $2\le g\le G$, then
$\ell_g\mid\ell_G$ and the same proof gives
\begin{equation}\label{eq:bounded-genus-lattice}
 \vol(\F)\in\ell_G^{-1}\ZZ_{\ge0}.
\end{equation}
In particular, there are only finitely many such volumes in a bounded interval.

\subsection{Finiteness of small volumes}\label{sec:geography}

The reduction to the Albanese fibration is essential: its genus is an invariant of the surface, whereas genera of arbitrary pencils on a fixed surface need not be bounded.

Let $S_0$ be the smooth minimal model of a surface $S$ of general type. Then $\vol(S)=K_{S_0}^2$ and $\chi(\OO_S)=\chi(\OO_{S_0})$. The Noether and Bogomolov--Miyaoka--Yau inequalities give
\[
 2\chi(\OO_S)-6\le\vol(S)\le9\chi(\OO_S).
\]
Consequently, bounded surface volume and bounded holomorphic Euler characteristic define the same boundedness condition, after changing the bound. In particular, $\vol(S)\le c$ implies
\[
 \chi(\OO_S)\le C(c):=\left\lfloor\frac{c+6}{2}\right\rfloor.
\]
We use the Euler characteristic formulation in the following lemma and then apply it with $C(c)$ to prove the volume formulation in the introduction; see also \cite{BHPV04}.
The next lemma is essentially due to Catanese \cite{Cat15}.

\begin{lemma}\label{lem:albanese-genus}
For every positive integer $c$, there is an integer $G(c)\ge2$ with
the following property. If $S$ is a smooth minimal surface of general
type with $\chi(\OO_S)\le c$ and its Albanese image is a curve, then
the genus $g\ge2$ of its Albanese fibration satisfies $g\le G(c)$.
\end{lemma}
\begin{proof}
The classical inequalities for minimal surfaces of general type give
\[
 1\le\chi(\OO_S)\le c,\qquad 1\le K_S^2\le9\chi(\OO_S)\le9c.
\]
\mathcorrection{There are therefore only finitely many pairs
$(\chi(\OO_S),K_S^2)$. For each pair, the Gieseker moduli space of
canonical models is of finite type, hence has finitely many connected
components. Deformation equivalence of canonical models agrees with
deformation equivalence of their smooth minimal resolutions
\cite[Theorems~166 and~167]{Cat15}. Thus only finitely many smooth
deformation types occur.} It suffices to show
that the Albanese fiber genus, whenever the Albanese image is a curve,
is determined by the integral cohomology ring and the canonical class.

Let $f:S\to B$ be the Albanese fibration, with general fiber $F$.
Since $f$ has connected fibers, the induced map
$\pi_1(S)\to\pi_1(B)$ is surjective.
It follows that the map on first homology is
surjective, and remains surjective after quotienting by torsion.
Since $H_1(B,\ZZ)$ is free, this surjection splits as a map of
abelian groups. Dualizing shows that $f^*H^1(B,\ZZ)$ is a primitive
sublattice of $H^1(S,\ZZ)$. Since $f$ is the Stein factorization of the Albanese map, every holomorphic one-form on $S$ is pulled back from $B$. Conversely, pullback of one-forms is injective, so $q(S)=g(B)$.
Both lattices therefore have rank $2q(S)=2g(B)$,
so
\[
 f^*:H^1(B,\ZZ)\xrightarrow{\ \sim\ }H^1(S,\ZZ).
\]
The cup product on a curve of positive genus surjects onto
$H^2(B,\ZZ)$. Consequently,
\begin{equation*}
 \operatorname{im}\!\left(
 \bigwedge\nolimits^2H^1(S,\ZZ)\xrightarrow{\ \smile\ }
 H^2(S,\ZZ)/\mathrm{tors}
 \right)=\ZZ[F].
\end{equation*}
The equality concerns the image subgroup itself and does not require
$[F]$ to be primitive. In particular, the argument also applies in the
presence of multiple fibers.
The integral cohomology ring determines the generator $[F]$ up to
sign. Since $F^2=0$, adjunction gives
\[
 K_S\cdot F=2g-2>0,
\]
which determines the sign and the genus $g$. In a smooth deformation
both the integral cohomology ring and the canonical class are preserved.
The genus is therefore constant for surfaces in the same deformation
class whose Albanese images are curves. Taking the maximum over the finitely
many deformation types proves the lemma.
\end{proof}

\begin{proof}[Proof of \autoref{thm:finite-intro}]
It is enough to consider foliations with $0<\vol(\F)<1/6$.
By \autoref{cor:alb-intro}, the foliation is induced by the Albanese fibration of $S$.
As $\F$ is of general type, the Albanese fibration is non-isotrivial of genus at least two.
In particular, the underlying surface $S$ is also of general
type. Passing
to the minimal model $S_0$ preserves the surface volume, the Euler characteristic, and the general Albanese fiber. Since $\vol(S)\le c$, the preceding inequalities give $\chi(\OO_{S_0})\le C(c)$. By \autoref{lem:albanese-genus}, the fiber genus is at most $G_c:=G(C(c))$.

According to \eqref{eq:bounded-genus-lattice}, every volume in $(0,1/6)$ belongs to
the set
\begin{equation*}
 \frac1{\ell_{G_c}}\ZZ\ \cap\ (0,1/6),
\end{equation*}
which is finite. Since $G_c\ge2$, one has $1/\ell_{G_c}\le1/6$.
Foliations of volume at least $1/6$ satisfy this bound as well. Thus
\[
 \varepsilon_c=\frac1{\ell_{G_c}}.\qedhere
\]

\end{proof}


\begin{remark}
	The above proof gives the existence of $G_c$ but does not give an explicit bound.
	Under the additional hypotheses considered by Ling--L\"u in
	\cite[Theorem~1.2]{LLlow} and \cite[Theorem~1.2]{LLhyp}, the genus $g$
	of the Albanese fiber is bounded explicitly in terms of $\chi(\OO_S)$.
	Combining their bounds with \autoref{prop:denominator} and
	\autoref{cor:alb-intro} therefore gives an explicit positive lower
	bound for $\vol(\F)$ in terms of $\chi(\OO_S)$ under those hypotheses.
\end{remark}

\subsection{A uniform lower bound problem}\label{sec:remaining}

For volumes strictly between zero and $1/6$, \autoref{cor:alb-intro}
reduces the irregular case to non-isotrivial Albanese fibrations.
In fact, a uniform positive lower bound for foliations of general type
on irregular surfaces exists if and only if such a bound exists for
algebraically integrable foliations of general type on smooth projective
surfaces. Here no bound on the surface volume, equivalently on $\chi(\OO_S)$ for surfaces of general type, is imposed.
Indeed, a lower bound $\varepsilon_{\rm alg}>0$ in the algebraically
integrable case gives the lower bound
$\min\{1/6,\varepsilon_{\rm alg}\}$ on irregular surfaces.
Conversely, suppose $\varepsilon_{\rm irr}>0$ bounds the positive volumes
on all irregular surfaces. Resolve the first integral of an algebraically
integrable foliation to a fibration $f:X\to B$ with reduced induced
foliation. If $g(B)>0$, the surface is already irregular. If
$B=\PP^1$, take a double cover $E\to\PP^1$ branched at four smooth fiber
values of $f$. Then $E$ is elliptic, and the fiber product $X_E$ is smooth
and irregular. The induced foliation $\F_E$ remains reduced, and, for
$\Pi:X_E\to X$,
\[
 K_{\F_E}=\Pi^*K_{\F},\qquad
 \vol(\F_E)=2\vol(\F).
\]
The first equality follows because the ramification is along smooth
invariant fibers.
Consequently,
$\vol(\F)\ge\varepsilon_{\rm irr}/2$, which proves the equivalence.

The part not covered by this reduction consists of transcendental
foliations on surfaces with $q(S)=0$. Rational surfaces form a natural
first class to consider.

The family in \autoref{thm:plane-family} in the next section also leaves open the questions
of which volume values in $(1/6,1/2)$ occur for foliations with positive
geometric genus and whether a positive accumulation point below $1/2$
can occur in this class.
When $\vol(\F)=1/6$ and $p_g(\F)>0$,
one may ask which configurations of invariant curves $C$ with $P\cdot C=0$
can occur and which values of the genus and self-intersection of the
transverse curve are possible.

\section{Examples}\label{sec:examples}

In this section, we construct two examples reaching the equality in \autoref{thm:lower-bound}.
The two examples also realize both cases of
\autoref{thm:equality-intro}. \autoref{thm:plane-family} also shows
that $1/2$ is an accumulation point of volumes of degree-two foliations
with geometric genus $p_g(\F)=1$.
The case $n=2$ realizes the transverse equality configuration, while \autoref{ex:invariant-canonical} realizes the invariant configuration of \autoref{thm:equality-core}.

\begin{example}\label{thm:plane-family}\label{sec:plane-family}
For every integer $n\ge2$, we construct a transcendental foliation
$\F_n$ of degree two on $\PP^2$ such that
\begin{equation*}
 p_g(\F_n)=1,\qquad
 \vol(\F_n)=\frac{n-1}{2(n+1)}.
\end{equation*}
\end{example}

For each $n\ge2$, choose an irrational real number $\alpha$ satisfying
\begin{equation*}
 1<\alpha<\frac{n}{n-1}.
\end{equation*}
Let $\F_n$ be the foliation defined on the affine chart $Z=1$ by
\begin{equation*}
 V=x(1-nx+\alpha y)\frac{\partial}{\partial x}
   +y\bigl(2-(n-1)x+y\bigr)\frac{\partial}{\partial y}.
\end{equation*}
The corresponding homogeneous one-form is
\begin{equation*}
\begin{split}
 \Omega={}&ZY\bigl(2Z-(n-1)X+Y\bigr)dX
 -ZX(Z-nX+\alpha Y)dY\\
 &+XY\bigl(-Z-X+(\alpha-1)Y\bigr)dZ.
\end{split}
\end{equation*}
\correction{Its coefficients have degree three, have no common factor, and satisfy
the Euler relation.} Thus $\F_n$ has degree two and canonical bundle
$\OO_{\PP^2}(1)$ on the plane. The three coordinate lines are invariant.

The only non-reduced singularities are $O=(0,0)$ and $U=[1:0:0]$,
with eigenvalues $(1,2)$ and $(1,n)$, respectively.
Indeed, the other affine singularities are
\[
 R=(1/n,0),\qquad T=(0,-2),\qquad Q=(x_Q,y_Q),
\]
where
\[
 x_Q=\frac{2\alpha-1}{(n-1)\alpha-n}<0,
 \qquad y_Q=(n-1)x_Q-2<0.
\]
The eigenvalues at $R$ and $T$ are $(-1,(n+1)/n)$ and
$(1-2\alpha,-2)$, respectively. At $Q$, the determinant of the
linear part is $x_Qy_Q\bigl((n-1)\alpha-n\bigr)<0$.
The other singularities at infinity are
\[
 W=[0:1:0],\qquad J=[1:1/(\alpha-1):0],
\]
with eigenvalues, up to a common non-zero factor,
\[
 (\alpha-1,-1),\qquad
 \left(-1,n-\frac{\alpha}{\alpha-1}\right),
\]
respectively. Hence these five singularities are reduced.
The irrational ratio of the eigenvalues at $W$ also shows that $\F_n$
is transcendental, since a reduced foliation induced by a fibration
has negative rational eigenvalue ratios at its singularities.

At both $O$ and $U$, the two coordinate lines are smooth transverse
invariant branches. 
Thus both singularities are analytically linearizable. Indeed, after scaling the vector field, a nonlinear resonant Poincar\'e--Dulac normal form with eigenvalues $(1,k)$, $k\ge2$, has the form
\[
x\partial_x+(ky+cx^k)\partial_y,\qquad c\ne0,
\]
and has only one formal separatrix. This contradicts the existence of the two smooth transverse invariant branches.
Their reductions require two and $n$ blow-ups, respectively.
In each sequence the last exceptional curve is non-invariant and the
preceding exceptional curves are invariant. Let $X$ be the resulting
reduced surface, and use $\F_n$ also for the induced foliation on $X$.
Write $H$ for the pullback of a line, and
\[
 e_1,e_2\quad\text{over }O,\qquad
 f_1,\ldots,f_n\quad\text{over }U
\]
for the total transforms of the exceptional divisors. The intersection
matrix in this basis is diagonal, with $H^2=1$ and $e_i^2=f_j^2=-1$.
By the transformation formula for the canonical divisor
\cite[Section~2.2]{LT24},
\begin{equation*}
 K_{\F_n}\sim H-e_2-f_n.
\end{equation*}
A linear form defines a canonical section on $X$ precisely when it
vanishes at $O$ and $U$. Thus the space of canonical sections is
$\CC\cdot Y$, and $p_g(\F_n)=1$.

Set
\[
 A=e_2,\qquad G=e_1-e_2,\qquad
 J_j=f_j-f_{j+1}\quad(1\le j<n).
\]
The strict transform $L$ of $Y=0$ has class $H-e_1-e_2-f_1$.
Consequently, the unique effective canonical divisor is
\begin{equation*}
 D=L+G+A+\sum_{j=1}^{n-1}J_j.
\end{equation*}
\correction{The curve $A$ is a smooth rational curve with $A^2=-1$ and is everywhere
transverse to $\F_n$.} It meets $G$ and $L$ once each.
The curves $G$ and $J_{n-1}+\cdots+J_1+L$ form two disjoint invariant
chains of $(-2)$-curves, with determinants $2$ and $n+1$.
The Zariski decomposition of $D$ is $D=P+N$, where
\begin{align*}
 N&=\frac12G+\frac1{n+1}L
       +\sum_{j=1}^{n-1}\frac{j+1}{n+1}J_j,\\
 P&=A+\frac12G+\frac n{n+1}L
       +\sum_{j=1}^{n-1}\frac{n-j}{n+1}J_j.
\end{align*}
Indeed, the support of $N$ has negative definite intersection matrix,
and direct computation gives
\[
 P\cdot G=P\cdot L=P\cdot J_j=0,\qquad
 P\cdot A=-1+\frac12+\frac n{n+1}
          =\frac{n-1}{2(n+1)}>0.
\]
Since $P$ is effective, it has non-negative intersection with every
irreducible curve outside its support. Thus $P$ is nef, and
\[
 \vol(\F_n)=P^2=P\cdot D=P\cdot A=\frac{n-1}{2(n+1)}.
\]
The other non-invariant exceptional curve $f_n$ meets $J_{n-1}$ once
and no other component of $D$. Therefore
\begin{equation*}
 P\cdot f_n=\frac1{n+1}.
\end{equation*}

These configurations are preserved on a relatively minimal model.
Indeed, an $\F_n$-exceptional curve $E$ with $K_{\F_n}\cdot E=-1$
would be a component of $N$, whose components all have self-intersection
$-2$. Hence $K_{\F_n}\cdot E=0$, and nefness gives
$P\cdot E=N\cdot E=0$. The curve $E$ cannot be a component of $P$:
the only component outside $N$ is the transverse curve $A$, whereas
$E$ is invariant. Since every coefficient of $P$ on its support is
positive, $E$ is disjoint from $\Supp(P)$. The same argument applies
to each subsequent contraction. Both $A$ and $f_n$ survive, and their
intersections with the nef part are preserved by the projection formula.
For $n=2$, the divisor $D'=A$ gives case~\textup{(ii)} of
\autoref{thm:equality-core} with $k=1$ and volume $1/6$.
As $n\to\infty$, the volumes increase to $1/2$.

\begin{example}\label{ex:invariant-canonical}
We construct a transcendental foliation $\F_{1/6}$ of degree two on
$\PP^2$ with $p_g(\F_{1/6})=1$ and $\vol(\F_{1/6})=1/6$, whose
effective canonical divisor on a reduced model is invariant.
This gives case~\textup{(i)} of \autoref{thm:equality-core}.
\end{example}

Consider the homogeneous one-form
\begin{equation*}
 \Omega=X^2Y\,dX-(X^3+XYZ+YZ^2)\,dY+Y^2(X+Z)\,dZ.
\end{equation*}
Its coefficients have no common factor and satisfy the Euler relation,
so it defines a foliation $\F_{1/6}$ of degree two. On $Z=1$, it is
generated by
\begin{equation*}
 V_0=\bigl(x^3+(1+x)y\bigr)\frac\partial{\partial x}
      +x^2y\frac\partial{\partial y}.
\end{equation*}
The only singularities are $O=[0:0:1]$ and $P_\infty=[0:1:0]$.
The linear part at $O$ is $y\frac{\partial}{\partial x}$.
At $P_\infty$, with $u=X/Y$ and $v=Z/Y$, a vector field is
$(u+v)\frac{\partial}{\partial u}-u^2\frac{\partial}{\partial v}$.
Thus $P_\infty$ is a reduced saddle-node.

The line $L=(y=0)$ is invariant. Blow up $O$, and then twice at the
intersection of the most recent exceptional curve with the strict
transform of $L$. In the successive charts $y=xt$, $t=xs$, and $s=xu$,
the induced foliations are generated by
\begin{align}
 V_1&=x\bigl(t(1+x)+x^2\bigr)\frac{\partial}{\partial x}
          -t^2(1+x)\frac{\partial}{\partial t},\notag\\
 V_2&=x(s+x+sx)\frac{\partial}{\partial x}
          -s(2s+x+2sx)\frac{\partial}{\partial s},\notag\\
 \label{eq:V3}
 V_3&=x(u+1+ux)\frac{\partial}{\partial x}
          -u(3u+2+3ux)\frac{\partial}{\partial u}.
\end{align}
All three exceptional curves are invariant, and the vanishing orders
of the vector fields at the three centers are $1,2,2$.
Let $S_3$ be the resulting surface and $\F_3$ the induced foliation.
The transformation formula for the canonical divisor
\cite[Section~2.2]{LT24} gives
\begin{equation*}
 K_{\F_3}\sim H-e_2-e_3,
\end{equation*}
where $H$ is the pullback of a line and $e_i$ denotes the total transform
of the $i$-th exceptional divisor. Set
\begin{equation*}
 C_1=e_1-e_2,\quad C_2=e_2-e_3,\quad A=e_3,\quad
 \widetilde L=H-e_1-e_2-e_3.
\end{equation*}
These curves form the chain $C_1-C_2-A-\widetilde L$, with
self-intersection numbers $-2,-2,-1,-2$, respectively.
The singularities at $C_1\cap C_2$, $C_2\cap A$ and
$A\cap\widetilde L$ have eigenvalues $(-1,2)$, $(-2,3)$ and $(1,-2)$.
The only remaining non-reduced point is
\[
 q:\quad x=0,\quad u=-\frac23.
\]

We check that $q$ has nonlinear Poincar\'e--Dulac normal form, so that
its further reduction is crepant. Put $z=u+2/3$. Equation~\eqref{eq:V3}
becomes
\begin{equation*}
 \dot x=\frac x3(3xz-2x+3z+1),\qquad
 \dot z=-\frac{3z-2}{3}(3xz-2x+3z).
\end{equation*}
The eigenvalues are $1/3$ and $2$. A second formal separatrix
transverse to $A=(x=0)$ would have the form
$z=\phi(x)=\sum_{j\ge1}a_jx^j$ and satisfy
\begin{equation*}
 -(3\phi-2)(3x\phi-2x+3\phi)
 =x\phi'(x)(3x\phi-2x+3\phi+1).
\end{equation*}
Comparing coefficients through degree five gives
\[
 a_1=\frac45,\quad a_2=-\frac{22}{25},\quad
 a_3=\frac{344}{375},\quad a_4=-\frac{118}{125},\quad
 a_5=\frac{2972}{3125}.
\]
The equation at degree six is $128/3125=0$, a contradiction.
Thus the normal form is nonlinear.
\begingroup
After scaling the vector field and making a local analytic change of
coordinates, its Poincar\'e--Dulac normal form is
\[
 W_6=x\frac{\partial}{\partial x}
       +(6y+cx^6)\frac{\partial}{\partial y},\qquad c\ne0.
\]
For $n\ge2$, blowing up the non-reduced point of
$W_n=x\partial_x+(ny+cx^n)\partial_y$ in the chart $y=xv$ gives
\[
 W_{n-1}=x\frac{\partial}{\partial x}
       +\bigl((n-1)v+cx^{n-1}\bigr)\frac{\partial}{\partial v}.
\]
The other singularity on the exceptional curve is reduced, with
eigenvalues $n$ and $1-n$. After five such blow-ups, the remaining
non-reduced point has generator
$W_1=x\partial_x+(y+cx)\partial_y$.
Blowing it up in the chart $x=uv$, $y=v$ gives
\[
 -cu^2\frac{\partial}{\partial u}
       +v(1+cu)\frac{\partial}{\partial v},
\]
which has a reduced saddle-node at $(0,0)$; the other chart introduces
no further non-reduced point. At all six centers the saturated vector
field has vanishing order one, and every new exceptional curve is
invariant. The canonical blow-up formula
$K_{\F'}=\sigma^*K_{\F}+(1-\operatorname{ord}_pW)E$
for an invariant exceptional curve therefore has discrepancy zero
at each step; see also \cite[Section~2.2]{LT24}.
\endgroup
Consequently, the reduction $\rho:Y\to S_3$ satisfies
\begin{equation}\label{eq:crepant-reduction}
 \mathcorrection{K_{\F_Y}=\rho^*K_{\F_3}.}
\end{equation}
The saddle-node also shows that the foliation is transcendental.

A section of $\OO_{\PP^2}(1)$ is represented on $Z=1$ by $ax+by+c$.
To give a section of $K_{\F_3}$, it must satisfy $c=0$ \correction{at the second
blow-up.} After division by a local equation of the exceptional
divisor, the section becomes $a+bxs$. \correction{The third blow-up then}
requires $a=0$. Consequently,
\[
 H^0(S_3,K_{\F_3})=\CC\cdot y,
\]
and \eqref{eq:crepant-reduction} gives $p_g(\F_{1/6})=1$.
The unique effective canonical divisor on $S_3$ is
\begin{equation*}
 D=\widetilde L+C_1+C_2+A.
\end{equation*}
Its Zariski decomposition is $D=P+N$, where
\begin{align*}
 N&=\frac23C_1+\frac13C_2+\frac12\widetilde L,\\
 P&=\frac13C_1+\frac23C_2+\frac12\widetilde L+A
    =\frac16(3H-e_1-e_2-e_3).
\end{align*}
Indeed, the intersection matrix of $\Supp(N)$ is negative definite,
and
\[
 P\cdot C_1=P\cdot C_2=P\cdot\widetilde L=0,\qquad
 P\cdot A=\frac16.
\]
Since $P$ is effective, these equalities imply that it is nef.
The crepant reduction pulls back the Zariski decomposition, and hence
\begin{equation*}
 \vol(\F_{1/6})=P^2=\frac16.
\end{equation*}
The strict transform of $A$ has intersection $1/6$ with $\rho^*P$.
All components of $D$ and all exceptional curves of $\rho$ are invariant,
so the effective canonical divisor on $Y$ is invariant as required.

\end{document}